\documentclass[11pt]{article}
\usepackage{amsmath}
\usepackage{amsthm}
\usepackage{amssymb}

\newtheorem{thm}{Theorem}
\newtheorem{theo}{Theorem}[section]
\newtheorem{definition}{Definition}[section]

\newtheorem{lemma}[theo]{Lemma}

\newtheorem{conje}[thm]{Cojecture}

\begin{document}

\title{
Crowns in pseudo-random graphs and Hamilton cycles in their squares
}

\author{
Michael Krivelevich
\thanks{
School of Mathematical Sciences,
Tel Aviv University, Tel Aviv 6997801, Israel.
Email: {\tt krivelev@tauex.tau.ac.il}.  Research supported in part
by USA-Israel BSF grant 2018267.}
}

\maketitle
\begin{abstract}
A crown with $k$ spikes is an edge-disjoint union of a cycle $C$ and a matching $M$ of size $k$ such that each edge of $M$ has exactly one vertex in common with $C$. We prove that if $G$ is an $(n,d,\lambda)$-graph with $\lambda/d\le 0.001$ and $d$ is large enough, then $G$ contains a crown on $n$ vertices with $\lfloor n/2\rfloor$ spikes. As a consequence, such $G$ contains a Hamilton cycle in its square $G^2$.
\end{abstract}

\section{Introduction}\label{s1}
A {\em Hamilton cycle} in a graph $G$ is a a cycle passing through all vertices of $G$. A graph possessing a Hamilton cycle is called {\em Hamiltonian}. Hamiltonicity has long been one of the most central and widely studied topics in Graph Theory; we refer the reader to surveys \cite{Gou14, KO14} reflecting the state of affairs at large.

Given the prominence of Hamiltonicity in the research in graphs, it is only natural to expect it to be studied in the context of pseudo-random graphs. Informally speaking, a graph $G$ on $n$ vertices is {\em pseudo-random} is its edge distribution resembles that of a truly random graph $G(n,p)$ of the same expected density $p=|E|/\binom{n}{2}$. The reader is invited to consult the survey \cite{KS06} for a comprehensive coverage of pseudo-random graphs.

Here we adopt the very frequently used formalism of $(n,d,\lambda)$-graphs as a model of pseudo-random graphs. A graph $G$ is an {\em $(n,d,\lambda)$-graph} if it has $n$ vertices, is $d$-regular, and all eigenvalues $\lambda_i$ of its adjacency matrix, but the first/trivial one $\lambda_1=d$, satisfy $|\lambda_i|\le \lambda$. Details about this model of pseudo-random graphs and its many properties can be found in \cite{KS06}. The so-called Expander Mixing Lemma (see Section \ref{s2-1} for a statement) provides a bridge between graph eigenvalues of regular graphs and their edge distribution.

One of the well known conjectures about pseudo-random graphs was proposed by the author and Sudakov some twenty years ago:
\begin{conje}[\cite{KS03}]\label{conj1}
There exists an absolute constant $C>0$ such that any $(n,d,\lambda)$-graph $G$ with $d/\lambda\ge C$ contains a Hamilton cycle.
\end{conje}
The validity of this conjecture would be very handy in proving Hamiltonicity of many regular graphs.

So far there have been several partial results towards establishing Conjecture \ref{conj1}. In the very same paper \cite{KS03}, Krivelevich and Sudakov proved that if $n$ is sufficiently large, then assuming
\begin{equation}\label{lb1}
\frac{d}{\lambda}\ge \frac{1000 \log n\log\log n}{(\log\log n)^2}
\end{equation}
guarantees Hamiltonicity of any $(n,d,\lambda)$-graph $G$. This shows the validity of Conjecture \ref{conj1} up to factors logarithmic in the number of vertices $n$. Hefetz, Krivelevich and Szab\'o \cite{HKS09} provided a sufficient condition for Hamiltonicity for general graphs in terms of their expansion and connectivity; when applied to $(n,d,\lambda)$-graphs, the condition reduces to (\ref{lb1}), up to multiplicative constants. Very recently, Glock, Munh\'a Correia and Sudakov \cite{GMS23} improved it further and showed that assuming $d/\lambda\ge C(\log n)^{1/3}$ for some large enough constant $C$ suffices to guarantee Hamiltonicity. They also established the validity of Conjecture \ref{conj1} for regular graphs of polynomially large degree $d\ge n^{\alpha}$ for an arbitrary fixed $\alpha>0$.

In this paper we make yet another step towards settling Conjecture \ref{conj1} and prove:
\begin{thm}
\label{th1}
Let $G$ be an $(n,d,\lambda)$-graph. If $\lambda/d\le 0.001$ and $d$ is large enough, then the square $G^2$ of $G$ contains a Hamilton cycle.
\end{thm}
\noindent(As usually, the {\em square} $G^2$ of a graph $G$ is defined as a graph on the same vertex set as $G$, where two vertices $u,v\in V(G)$ are connected by an edge if their distance in $G$ is at most two.)

In order to prove the above theorem, we will argue about the existence of special structures in $G$ itself. Define a {\em crown with $k$ spikes} to be an edge-disjoint union of a cycle $C$ and a matching $M$ of size $k$ where each edge of $M$ has exactly one vertex in common with $C$. Observe that if $H$ is a crown with cycle $C=(v_1,\ldots,v_{\ell},v_1)$ and matching $M=\{v_{i_1}u_1,\ldots, v_{i_k}u_k\}$, then for every spike $v_{i_j}u_j$ the distance in $H$ between $u_j$ and $v_{i_j+1}$ is two. Hence inserting  $u_j$ right after $v_{i_j}$ in the natural order of the cycle stated above produces a cyclic order $\sigma$ of all vertices of $H$ where every pair of adjacent vertices are at distance at most two in $H$. This order certifies the Hamiltonicity of $H^2$. Hence, Theorem \ref{th1} follows immediately from:

\begin{thm}
\label{th2}
Let $G$ be an $(n,d,\lambda)$-graph, with $\lambda/d\le 0.001$ and $d$ large enough. Then $G$ contains a crown on $n$ vertices with $\lfloor n/2\rfloor$ spikes.
\end{thm}
\noindent(Obviously finding a spanning crown in $G$ with an arbitrary number of spikes suffices to derive that $G^2$ is Hamiltonian.)

The constant $0.001$ in the above statements is certainly suboptimal and can be tightened through a more careful (and probably more tiring, both for the author and the reader) implementation of the same arguments. We however see little point in fighting for its possible value here, preferring simplicity and readability instead.

The notation used in the paper is fairly standard/self-explanatory. In particular, for a graph $G=(V,E)$ and disjoint vertex subsets $U,W\subset V$ we denote by $N(U,W)$ the set of neighbors of $U$ in $W$ and set $N(U)=N(U,V\setminus U)$. We also denote by $\Gamma(U)$ the set of all vertices in $V$ (including those in $U$) having a neighbor in $U$. Similarly, for a subset $W$, not necessarily disjoint from $U$, we let $\Gamma(U,W)$ be all vertices in $W$ having a neighbor in $U$, so that $\Gamma(U)=\Gamma(U,V)$. We omit rounding signs systematically so as to improve readability.

We set
\begin{eqnarray*}
\delta &=& 0.001\,,\\
\delta_1 &=& 48\delta= 0.048\,.
\end{eqnarray*}

The paper is structured as follows. In Section \ref{s-outline} we provide an outline of our proof, hoping it would be helpful for the reader when parsing the main arguments. In the following section, Section \ref{s2}, we present a set of tools used in the proof of our main result. The proof of Theorem \ref{th2} is then given in Section \ref{s-proof}. The final section, Section \ref{s-concl}, is devoted to concluding remarks and open questions.

This work draws its inspiration, both ideological and technical, from several prior results about embedding problems in random and pseudo-random graphs, most notably \cite{AKS07,Mon19,DKN22}.

\section{Outline of the proof}\label{s-outline}
Let us outline our argument in this section. At the last step of the proof we will be looking for a nearly perfect matching $M$ between an already constructed cycle $C$ of length $\lceil n/2\rceil$ and its complement $V-C$; the matching's edges will serve as the spikes of the crown. If we are to hope this matching exists, we need to ensure that the degrees on both sides of the bipartite graph between $C$ and $V-C$ are positive, or better yet, the relevant bipartite graph is a reasonably good expander. Taking care of side $C$ is relatively easy --- we just find and put aside a vertex subset $S_1$ of small linear size so that every vertex in the graph has $\Theta(d)$ neighbors in $S_1$. (Adopting Montgomery's nice terminology from \cite{Mon19}, we call $S_1$ and other sets of a similar type {\em matchmakers}.) Arguing about the existence of such a set is done through the Local Lemma, see Section \ref{s2-3}. We then embed $C$ in $V-S_1$, thus making sure the set $S_1$ stays outside of $C$. The vertices of $V-S_1$ not used eventually in the embedding of $C$ are then released to the other part of the bipartite graph.

Taking care of positive degrees of the vertices of $V-C$ into $C$ is more challenging. In order to achieve this goal, we get another matchmaker set $S_2$, disjoint from $S_1$ and again with $\Theta(d)$ neighbors for every vertex of the graph, and make sure this set is fully immersed in cycle $C$. For this to happen, we first embed $S_2$ into a path $P_1$ of length $O(\delta)n$ inside $V-S_1$. To accomplish the latter task, we use a very powerful and flexible embedding technique based on the Friedman-Pippenger theorem \cite{FP87}, combined with the idea of rollbacks  due to Johannsen \cite{Johan}, described also in the PhD thesis of Glebov \cite{Gle13} and  attributed there to Glebov, Johannsen and the author of the present paper. This combination allows to gradually embed large trees in expanding graphs starting with a good embedding of the empty forest of $|S_2|$ vertices. At each iteration yet another part of the tree is embedded, then a connecting edge between appropriate parts of the embedded piece is found. Then we roll back vertex by vertex large parts of the embedded piece not used in the embedding, while still preserving the quality and extendability of the whole embedded structure.  We do it pretty much in the style of \cite{DKN22}, see Section \ref{s2-4} for the formal description of the relevant tools. In order to pull this all through, we need to guarantee that the induced subgraph $G[V-S_1]$ is a good expander, including expansion of the set $S_2$ and its subsets into $V-S_1$. To ensure the required expansion property, we create and use yet another matchmaker $S_3$, disjoint from $S_1$ and $S_2$ and again well connected to every vertex of the graph.

Once we find such a path $P_1$ of length $\ell_1$ containing $S_2$, we extend it to what we call a double broom $T_{\ell_1,t}$, which is a path of length $\ell_1$ with two disjoint complete binary trees of depth $t$, each attached to the endpoint of $P_1$; this tree has $O(\delta)n$ vertices. The same embedding technique based on Friedman-Pippenger is employed here. Then we find another double broom $T_{\ell_2,t}$ in $V-S_1$ and disjoint from  $T_{\ell_1,t}$; standard approaches to embedding large trees in expanding graphs \cite{Hax01,BCPS10} are invoked, see Section \ref{s2-5} for a formal statement. This step requires some cleaning of the subgraph at hand in order to find a large induced expander in $V-S_1-V(T_{\ell_1,t})$; this is pretty straightforward and is done formally through the statement in Section \ref{s2-2}.

The parameters of the second tree $T_{\ell_2,t}$ are chosen so that $\ell_1+\ell_2+4t+2=\lceil n/2\rceil$. If this is the case, then finding two edges of $G$ connecting the left, resp. right, broom of the first tree with the left, resp. right, broom of the second tree closes a cycle of length exactly $\lceil n/2\rceil$. This is because in the double broom with parameters $\ell$ and $t$ every leaf of the first attached tree is connected by a path of length $\ell+2t$ to every leaf of the second attached tree. The required two edges exist due to standard edge distribution properties of the pseudo-random graph derived from the Expander Mixing Lemma, see Section \ref{s2-1}. This way we obtain a cycle $C$ of the required length, and containing the second matchmaker $S_2$ in full, as desired.

Now we are finally back to arguing about the existence of a matching $M$ of size $|M|=\lfloor n/2\rfloor$ between  $C$ and $V-C$. We arrived at this point fully prepared, having matchmakers $S_1$ and $S_2$ on either side of the bipartite graph. These matchmakers guarantee expansion of relatively small sets due to the minimum degree; larger sets take care of themselves as any set of size $\delta n$ in $G$ sees all but less than $\delta n$ vertices outside. The details are handled using the tools in Section \ref{s2-1}. These expansion properties suffice to find a desired matching, thus attaching the required spikes to our crown and completing the proof.

\section{Main tools}\label{s2}
In this section we gather main tools to be applied in the proof. Throughout this section, we assume that $G=(V,E)$ is an $(n,d,\lambda)$-graph with $\lambda/d\le \delta$.

We will use the following standard definition of expansion.
\begin{definition}\label{def:expanding}
	Let $s \in \mathbb{N}$ and $K > 0$. We say that a graph $G$ is \emph{$(s, K)$-expanding} if for every subset $X \subseteq V(G)$ of size $|X| \le s$ we have $|N(X)| \ge K |X|$.
\end{definition}

\subsection{Expander Mixing Lemma and its consequences}\label{s2-1}
In order to bridge between graph eigenvalues and edge distribution in $(n,d,\lambda)$-graphs, we use the famed Expander Mixing Lemma due to Alon and Chung \cite{AC88} (see also Theorem 2.11 of \cite{KS06}). For every two (not necessarily disjoint) subsets $B,C\subseteq V$ , let $e(B,C)$ denote the number of ordered pairs $(u,v)$ with $u\in B,v\in C$ such that $uv$ is an edge of $G$. Note that if $u,v\in B\cap C$, then the edge $uv$ contributes 2 to $e(B,C)$. In this notation,
\begin{equation}\label{xml}
\left| e(B,C)-\frac{|B||C|d}{n}\right|<\lambda\sqrt{|B||C|}\,.
\end{equation}
We can derive from this statement quantitative estimates for the expansion of such $G$.

\begin{lemma}\label{le1}
For every $U\subseteq V(G)$ and for every $X\subset U$ with $|X|\ge \delta n$, one has: $|N(X,U)|> |U|-|X|-\delta n$.
\end{lemma}
\begin{proof}
Choose $X_0\subseteq X$, $|X_0|=\delta n$. Then for every $Y\subseteq U$ of cardinality $|Y|=\delta n$, we have by (\ref{xml}):
$$
e(X_0,Y)>\frac{|X_0||Y|d}{n}-\lambda\sqrt{|X_0||Y|}=\frac{\delta^2n^2d}{n}-\lambda\delta n\ge \delta^2dn-\delta^2dn=0\,,
$$
implying that $X_0$ has a neighbor in $Y$. It thus follows that $|N(X_0,U)|>|U|-|X_0|-\delta n$. Hence, $|N(X,U)|\ge |N(X_0,Y)|-|X-X_0|> |U|-|X|-\delta n$.
\end{proof}

The following lemma guarantees that induced subgraphs of $G$ with relatively large minimum degree are good expanders.
\begin{lemma}\label{le2}
Let $S\subseteq V$ be such that every vertex $v \in V$ has at least $d_1$ neighbors in $S$, with $d_1\ge 2\delta d$. Then every subset $X\subset V$ with $|X|\le \delta n$ satisfies: $|\Gamma(X,S)|\ge \frac{d_1}{2\delta d}|X|$.
\end{lemma}
\begin{proof}
Let $x=|X|\le \delta n$, and denote $Y=\Gamma(X,S)$; denote also $A=\frac{d_1}{2\delta d}$ and observe that $A\ge 1$.  If $|Y|< Ax$, then we have by (\ref{xml}):
$$
d_1x\le e(X,Y)< \frac{x\cdot Ax\cdot d}{n}+\lambda\sqrt{x\cdot Ax}\\,
$$
implying:
$$
d_1< \frac{Axd}{n}+\lambda\sqrt{A}\le \delta Ad+\delta\sqrt{A}d\le 2\delta Ad=d_1
$$
--- a contradiction.
\end{proof}

\subsection{Finding expanding subgraphs in large vertex subsets}\label{s2-2}
\begin{lemma}\label{le2-1}
Let $U\subseteq V$ be a subset of at least $n/2$ vertices. Then there exists a subset $U_0\subset U$, $|U_0|\ge |U|-\delta n$, such that for every $X\subset U_0$ with $|X|\le \delta n$, one has $|N(X,U_0)|\ge \frac{1-6\delta}{4\delta}|X|$.
\end{lemma}
\begin{proof}
The argument here is nearly identical to that of Lemma 4.1 in \cite{FK21}. Set $C=\frac{1-6\delta}{4\delta}$. We start with $W=\emptyset$, and for as long as there exists a subset $W_0\subset U\setminus W$ with $|W_0|\le \delta n$ and $|N(W_0,U\setminus W)|<C|W_0|$, we add $W_0$ to $W$. It is easy to see that at any point of this procedure we have $|N(W,U)|\le C|W|$. Assume $|W|$ reaches $\delta n$ at some point. Then at that point $\delta n\le |W|<2\delta n$, and by Lemma \ref{le1} we have $|N(W,U)|> |U|-|W|-\delta n$. It thus follows that
$$
|U|-|W|-\delta n < |N(W,U)|\le C|W|\,,
$$
implying $C>(|U|-|W|-\delta n)/|W|>(n/2-2\delta n-\delta n)/(2\delta n)= \frac{1-6\delta}{4\delta}$ --- a contradiction. Hence the above described cleaning process halts with $|W|\le \delta n$, and by its description the set $U_0=U\setminus W$ meets the lemma's requirement.
\end{proof}

\subsection{Splitting vertex degrees}\label{s2-3}
\begin{lemma}\label{le3}
For large enough $d$, the graph $G$ contains three vertex disjoint subsets $S_1,S_2,S_3$, each with at most $\delta_1n$ vertices, such that every vertex $v\in V$ has at least $\frac{\delta_1d}{4}$ neighbors in each of $S_i$'s.
\end{lemma}
\begin{proof}
This is a simple consequence of the Lov\'asz Local Lemma (see, e.g., Chapter 5 of \cite{AS16}). Color the vertices of $G$ randomly and independently in $K$ colors, with $K=2/\delta_1$. For $v\in V$ and color $i\in [K]$, let $A_{v,i}$ be the event that $v$ has less than $d/2K$ neighbors of color $i$. By a standard Chernoff-type inequality, $\Pr[A_{v,i}]\le e^{-d/8K}$. Also, the event $A_{v,i}$ is independent of all events $A_{u,j}$ but those for which $u=v$ or $u$ and $v$ have a common neighbor; the number of such events is less than $Kd^2$. Recalling that we assume $d$ to be large enough, it follows by the Local Lemma that with positive probability none of the events $A_{v,i}$ holds. Fix such a coloring, and let $S_1,S_2,S_3$ be its smallest color classes (breaking ties arbitrarily) with $|S_1|\le |S_2|\le |S_3|$. Then $|S_3|(K-2)\le n$, implying $|S_1|,|S_2|,|S_3|\le \frac{n}{K-2}<\frac{2n}{K}=\delta_1 n$.
\end{proof}

\noindent{\bf Remark.} The proof above uses only the assumption that $G$ is $d$-regular. Hence the conclusion of Lemma \ref{le3} is valid for any $d$-regular graph.

\subsection{Friedman--Pippenger with rollbacks}\label{s2-4}
As we have mentioned, the Friedman--Pippenger theorem \cite{FP87} and its rollback version \cite{DKN22} provide a very powerful tool for gradual embedding of large tree-like structures in expanding graphs. Below we introduce these tools. Our notation and presentation follow closely that of \cite{DKN22}.

We will need the notion of an \emph{$(s,D)$-good} embedding.
\begin{definition}\label{def:goodness}
Let $G$ be a graph and let $s, D \in \mathbb{N}$. Given a graph $F$ with maximum degree at most $D$, we say that an embedding $\phi \colon F \hookrightarrow G$ is \emph{$(s,D)$-good} in $G$ if
\begin{equation} \label{eq:extendable}
	|\Gamma_G(X) \setminus \phi(F)| \ge \sum_{v \in X} \left[ D - \deg_F(\phi^{-1}(v)) \right]+|\phi(F)\cap X|
\end{equation}
for every $X \subseteq V(G)$ of size $|X| \le s$. Here we slightly abuse the notation by setting $\deg_F(\emptyset) := 0$, i.e. if a vertex $v\in V(G)$ is not used by $\phi$ to embed $F$, then we set $deg_F(\phi^{-1}(v))=0$.
\end{definition}

\begin{theo}[\cite{FP87, DKN22}] \label{thm:FP}
	Let $F$ be a graph with $\Delta(F) \le D$ and $|V(F)| < s$, for some $D, s \in \mathbb{N}$. Suppose we are given a $(2s-2,D{+}2)$-expanding graph $G$ and a $(2s-2,D)$-good embedding $\phi \colon F \hookrightarrow G$. Then for every graph $F'$ with $|V(F')| \le s$ and $\Delta(F') \le D$ which can be obtained from $F$ by successively adding a new vertex of degree $1$, there exists a $(2s-2,D)$-good embedding $\phi' \colon F' \hookrightarrow G$ which extends $\phi$.
\end{theo}

The second result we need is a simple corollary of the definition of $(s, D)$-goodness. While easy to prove, this observation \cite{Johan,DKN22} turns out to yield a very powerful method for embedding problems in expanding graphs, where one typically embeds a larger structure than desired at each iteration, fixes a part of it, and then rolls it back by erasing gradually the unused part, before proceeding to the next iteration while keeping a good embedding. It has also been utilized by Montgomery in \cite{Mon19} for embedding spanning trees in random graphs.

\begin{lemma}[\cite{DKN22}] \label{lemma:delete}
	Suppose we are given graphs $G$ and $F$ with $\Delta(F)\leq D$, and an $(s, D)$-good embedding $\phi \colon F \hookrightarrow G$, for some $s, D \in \mathbb{N}$. Then for every graph $F'$ obtained from $F$ by successively removing a vertex of degree $1$, the restriction $\phi'$ of $\phi$ to $F'$ is also $(s, D)$-good in $G$.
\end{lemma}

\subsection{Embedding nearly spanning trees}\label{s2-5}
The Friedman-Pippenger theorem, Theorem \ref{thm:FP}, is very handy for embedding moderate linear sized  bounded degree trees in expanders. In many circumstances however one needs to embed nearly spanning trees. This task was addressed explicitly in \cite{AKS07}. Here we use the following generalization of Friedman-Pippenger due to Haxell \cite{Hax01}. We state a version of it as given in \cite{BCPS10}.
\begin{theo}[\cite{BCPS10}]\label{th-BCPS}
Let $d$, $m$ and $M$ be positive integers. Assume that $H$ is a non-empty graph satisfying
the following two conditions:
\begin{enumerate}
\item For every $U\subseteq V(H)$ with $0 < |U| \le m$, $|N_H(U)| \ge d |U| + 1$;
\item For every $U\subseteq V(H)$ with $m < |U| \le 2m$, $|N_H(U)| \ge d |U| + M$.
\end{enumerate}
Then $H$ contains every tree $T$ with $M$ vertices and maximum degree at most $d$.
\end{theo}

\section{Proof of Theorem \ref{th2}}\label{s-proof}
In this section we prove our main result, Theorem \ref{th2}. We follow the outline of the proof as presented in Section \ref{s-outline}.

\medskip

\noindent{\bf Stage 1: three matchmakers are born.}\\
We apply Lemma \ref{le3} to $G$ to find three disjoint subsets $S_1,S_2,S_3$ of sizes $|S_i|\le \delta_1n$, with every vertex of $G$ having at least $\frac{\delta_1d}{4}$ neighbors in each of $S_i$'s. The set $S_1$ will serve as a matchmaker for the vertices on the cycle $C$ in the crown, the set $S_2$ will be a matchmaker for the vertices outside of $C$, and finally the set $S_3$ will serve as an internal matchmaker ensuring the expansion in the subgraph spanned by $V\setminus S_1$.

\medskip

\noindent{\bf Stage 2: the second matchmaker gets swallowed by a serpent.}\\
 Denote
 $$V_1=V\setminus S_1\,.
 $$
 Our goal at this stage is to find a path (a serpent) $P_1$ of length $O(\delta)n$ in $V_1$ containing $S_2$ in full; moreover, since we are to extend this path into a double broom at the next stage, we need the eventual embedding of the path to be good in the sense of Definition \ref{def:goodness}. At this and the following stages we sometimes do not distinguish between a forest and its copy already embedded in $V_1$.

Observe first that by our choice of the set $S_3$ and by Lemma \ref{le2} with $d_1=\frac{\delta_1d}{4}=12\delta d$ and $S_3$ in place of $S$, we have that $|\Gamma(X,S_3)|\ge 6|X|$ for every subset $X\subset V_1$ of at most $\delta n$ vertices. Also, by Lemma \ref{le1}, for sets $X\subset V_1$ of cardinalities $\delta n<|X|\le \frac{(1-\delta)n}{6}$,  we have $|N(X,V_1)|\ge 5|X|$. Hence the graph $G[V_1]$ is a $\left(\frac{(1-\delta)n}{6},5\right)$-expander.

Let $F$ be a forest on $|S_2|$ vertices with no edges, and fix an arbitrary bijection $\phi_0: F \hookrightarrow S_2$. We first need to verify that $\phi_0$ is a $\left(\frac{(1-\delta)n}{6},3\right)$-good embedding in $G[V_1]$. Recalling Definition \ref{def:goodness}, it is enough to check that $|\Gamma(X,V_1\setminus S_2)|\ge 4|X|\ge 3|X|+|X\cap S_2|$ for every $X\subset V_1$ of size $|X|\le \frac{(1-\delta)n}{6}$. Consider first the case where $|X|\le \delta n$. There we get: $|\Gamma(X,V_1\setminus S_2)|\ge 6|X|$. For $\delta n<|X|\le \delta_1n$, we have  by Lemma \ref{le1}: $|N(X,V_1)|\ge |V_1|-|X|-\delta n\ge |V_1|-49\delta n\ge 6\delta_1n$. Finally, for $\delta_1n<|X|\le \frac{(1-\delta)n}{6}$, we obtain due to expansion: $|N(X,V_1)|\ge 5|X|$, implying $|\Gamma(X,V_1\setminus S_2)|\ge 4|X|$, as required (with room to spare). We notice with foresight that this should allow embedding of forests of size up to $\frac{(1-\delta)n}{12}$ in $V_1$, using the tools of Section \ref{s2-4}.

We now perform an iterative procedure, gradually merging the vertices of $S_2$ into one path. Let $|S_2|=s_2$. We initialize $F_{s_2}=F$, $E_{s_2}=\emptyset$. Now, for $i=s_2$ down to 2 we repeat as follows. Assume that at the beginning of step $i$ we have in $G[V_1]$ a family ${\cal P}$ of $i$ vertex disjoint paths covering $S_2$ with all endpoints in $S_2$, and a set $E_i$ of $|E_i|=s_2-i$ edges in the union of these paths such that deleting $E_i$ from the union of the paths in ${\cal P}$ gives a $\left(\frac{(1-\delta)n}{6},3\right)$-good embedding of forest $F_i$ in $G[V_1]$. We also assume that $|V(F_i)|\le s_2+\sum_{j=i+1}^{s_2}\max\left\{0,\log_2\frac{6\delta n}{i}\right\}$. These assumptions clearly hold for $i=s_2$. We choose one endpoint from every path in ${\cal P}$ to form a set $X_i$ of cardinality $|X_i|=i$. The set $X_i$ is split into nearly equal parts $X_{i1},X_{i2}$ of cardinalities at least $\lfloor i/2\rfloor$ each. We extend $F_i$ into a forest $F_{i+1}'$ as follows. We grow disjoint complete binary trees of height $t_i$ from every vertex of $X_i$, where
$$
t_i=\min\left\{t: \left\lfloor \frac{i}{2}\right\rfloor\cdot 2^t\ge \delta n\right\}\,;
$$
we have $t_i\le \max\left\{0,\log_2\frac{6\delta n}{i}\right\}$. It follows that $|V(F_{i+1}'|\le |V(F_i)|+4\delta n$. We assume -- and verify it later -- that $|V(F_{i+1}')|\le \frac{(1-\delta)n}{12}$. Also, $F_{i+1}'$ can be obtained from $F_i$ by successively adding a new vertex of degree 1. Hence by Theorem \ref{thm:FP} a $\left(\frac{(1-\delta)n}{6},3\right)$-good embedding of $F_i$ in $G[V_1]$ can be extended to a $\left(\frac{(1-\delta)n}{6},3\right)$-good embedding of $F_{i+1}'$. Observe that by the definition of $t_i$, the sets of leaves of all trees grown from $X_{ij}$, $j=1,2$, have at least $\delta n$ vertices each. Therefore, after the embedding of $F_{i+1}'$ in $G[V_1]$, the graph $G$ has an edge $e_i$ between (the images of) these two sets of leaves by Lemma \ref{le1}. This edge closes a path $P_i$ of length $2t_i+1$ between a vertex in $X_{i_1}$ and a vertex in $X_{i_2}$ in $G[V_1]$. We now add the edges of $P_i$ but $e_i$ to $F_i$ to form a new forest $F_{i+1}$. Finally, we roll back all other edges of these binary trees we constructed, as described in Lemma \ref{lemma:delete}, to obtain a $\left(\frac{(1-\delta)n}{6},3\right)$-good embedding of $F_{i+1}$ in $G[V_1]$. The path $P_i$ replaces then two paths in ${\cal P}$ whose endpoints are connected by $P_i$, thus reducing the total number of paths in ${\cal P}$ by one. We also add $e_i$ to $E_i$ to form $E_{i+1}$. This completes step $i$.

It remains to estimate from above the total number of vertices ever consumed by this embedding. (Notice that the number of vertices in the forests $F_i, F_i'$ goes up and down, due to our forth and back moves of embedding and rolling back.) We can cap this number by
\begin{eqnarray*}
&& s_2+2\sum_{i=2}^{s_2} \max\left\{0,\log_2\frac{6\delta n}{i}\right\} + 4\delta n\\
&\le& \delta_1n+2\sum_{i=2}^{6\delta n}\log_2\frac{6\delta n}{i}+4\delta n\\
&=&\delta_1n+12\delta n\cdot\log_2(6\delta n)-2\sum_{i=2}^{6\delta n}\log_2i+4\delta n\\
&\le& \delta_1n+\frac{12\delta n\cdot\ln(6\delta n)}{\ln 2}-2\left(\frac{6\delta n\ln(6\delta n)-6\delta n}{\ln 2}\right)+4\delta n\\
&=& \delta_1n+\frac{12\delta n}{\ln 2}+4\delta n<70\delta n<\frac{(1-\delta)n}{12}\,.
\end{eqnarray*}
This estimate certifies that the above described embedding procedure can indeed be pulled through using Theorem \ref{thm:FP} and Lemma \ref{lemma:delete}.

By the end of this stage, we have a path $P_1$ in $G[V_1]$ on length $\ell_1\le 70\delta n$ vertices with endpoints $a_1,b_1\in S_2$, and a set of edges $E_1$ of cardinality $|E_1|=s_2-1$ inside it, so that $S_2\subseteq V(P_1)$, and the forest $F_1$ obtained from $P_1$ by deleting the edges of $E_1$ (and with $V(F_1)=V(P_1)$) is a $\left(\frac{(1-\delta)n}{6},3\right)$-good embedding in $G[V_1]$.

\medskip

\noindent{\bf Stage 3: the serpent morphs into a double broom.}\\
A {\em double broom} $T_{\ell,t}$ with parameters $\ell$ and $t$ is obtained by joining the roots of two disjoint complete binary trees of depth $t$ by a path of length $\ell$. Our goal at this stage is to extend the path $P_1$ of length $\ell_1$ with endpoints $a_1,b_1$ into a double broom $T_{\ell_1,t}$ in $G[V_1]$. We set $t$ to be the minimal integer such that a complete binary tree of depth $t$ has at least $\delta n$ leaves. Then $T_{\ell_1,t}$ has at most $78\delta n$ vertices.

Recall that and the forest $F_1$ obtained from $P_1$ by deleting the edges of $E_1$ (and with $V(F_1)=V(P_1)$) is a $\left(\frac{(1-\delta)n}{6},3\right)$-good embedding in $G[V_1]$. Also, $G[V_1]$ is a $\left(\frac{(1-\delta)n}{6},5\right)$-expander.
The graph derived from $F_1$ by growing complete binary trees of depth $t$  at both $a_1$ and $b_1$, where the trees are disjoint from each other and from $V(F_1)=V(P_1)$, has maximum degree 3 and can be obtained from $F_1$ by successively adding vertices of degree 1.
Hence by Theorem \ref{thm:FP} the graph $G[V_1]$ contains a copy of this graph. Putting back the edges of $E_1$, we obtain a copy $T_1$ of the double broom $T_{\ell_1,t}$ in $G[V_1]$, with all vertices of $S_2$ contained in this copy (in fact in $P_1$).

\medskip

\noindent{\bf Stage 4: another double broom is created.}\\
Let $U=V_1-V(T_1)$. Then $|U|\ge n-\delta_1n-78\delta n=n-126\delta n$. Set
$$
\ell_2=\lceil n/2\rceil-\ell_1-4t-2\,.
$$
Our goal here is to find a copy of the double broom $T_{\ell_2,t}$ in $G[U]$.

By applying Lemma \ref{le2-1} to $G$ and $U$ we find a subset $U_0\subseteq U$ of size $|U_0|\ge |U|-\delta n\ge (1-127\delta)n>n/2$ such that every subset $X\subset U_0$ of at most $\delta n$ vertices satisfies: $|N(X,U_0)|\ge \frac{1-6\delta}{4\delta}|X|>5|X|$. In addition, every subset $X\subseteq U_0$ of at least $\delta n$ vertices sees at least $|U_0|-|X|-\delta n$ vertices outside of it in $U_0$ by Lemma \ref{le1}. Hence by applying Theorem \ref{th-BCPS} to $G[U_0]$ for $H$ and with $T=T_{\ell_2,t}$, $d=3$, $m=\delta n$ and $M=|V(T_{\ell_2,t})|\le n/2+4\delta n$, we find a copy $T_2$ of $T_{\ell_2,t}$ inside $U_0$.

\medskip

\noindent{\bf Stage 5: double brooms merge into a crown's ring.}\\
Notice that a double broom $T_{\ell,t}$ has the following nice feature: it has two subsets of leaves $L_1$ and $L_2$ (the leaves of the left brush and of the right brush, respectively) so that the tree contains a path of the same length $\ell+2t$ between every vertex in $L_1$ and every vertex in $L_2$. We will use this property to create a cycle $C$ of desired length from double brooms $T_1$ and $T_2$.

The trees $T_1$ and $T_2$ have two naturally defined disjoint sets of leaves each, $L_{11}$ and $L_{12}$ for $T_1$ and $L_{21}$ and $L_{22}$ for $T_2$, all these four sets are of cardinalities at least $\delta n$. Then by Lemma \ref{le1}, the graph $G$ has an edge $f_1$ between $L_{11}$ and $L_{21}$, and another edge $f_2$ between $L_{12}$ and $L_{2}$. These edges connect the endpoints of two disjoint paths in $G[V_1]$, one of length $\ell_1+2t$ in $T_1$ and another of length $\ell_2+2t$ in $T_2$. Hence they close a cycle $C$ of length $\ell_1+\ell_2+4t+2=\lceil n/2\rceil$ in $G[V_1]$. Since $C$ contains $P_1$, we have $C$ including the second matchmaker $S_2$ in full, as desired.

\medskip

\noindent{\bf Stage 6: spikes are attached to the ring --- and we have a crown!}\\
At the final stage of the proof we have a cycle $C$ of length exactly $\lceil n/2\rceil$ such that the first matchmaker set $S_1$ is completely outside $C$, while the second matchmaker set $S_2$ is completely immersed in $C$. Recall that every vertex of $G$ has at least $\delta_1 d/4$ neighbors in $S_1$ and in $S_2$. We need to find a matching of size $\lfloor n/2\rfloor$ between $C$ and $V\setminus V(C)$ in $G$.

As usually, it is enough to verify a Hall-type condition for this bipartite graph. Specifically, we can settle for checking that every set $X$ of size up to $\lceil n/4\rceil$ in one of the sides $V(C),V\setminus V(C)$ is connected to at least $|X|+1$ vertices on the other side. Applying Lemma \ref{le2} with $d_1=\delta_1d/4$ and $S=S_1$ and also $S=S_2$, we conclude that the latter condition holds true for sets $X$ of sizes $|X|\le \delta n$. (The sets $S_1,S_2$ finally perform their function of matchmakers here.) As for sets $X$ with $\delta n<|X|\le \lceil n/4\rceil$, such sets see at least $\lfloor n/2\rfloor-\delta n>|X|+1$ vertices on the other side of the bipartite graph due to Lemma \ref{le1}.

If follows by Hall's Theorem that $G$ contains a matching $M$ of size $|M|=\lfloor n/2\rfloor$ between $V(C)$ and $V\setminus V(C)$. Adjoining this matching to cycle $C$ produces a required crown. The proof is complete.\hfill$\Box$

\section{Concluding remarks}\label{s-concl}
In this paper, we made progress towards the well known conjecture of the author and Sudakov about Hamiltonicity of pseudo-random graphs and proved that an $(n,d,\lambda)$-graph $G$ with $d/\lambda\ge 1000$ contains a Hamilton cycle in its square $G^2$. This was achieved through embedding of a crown on $n$ vertices with $\lfloor n/2\rfloor$ spikes in $G$.

Essentially the same proof, mutatis mutandis, gives also that for any $\epsilon>0$ there exists $C>0$ such that an $(n,d,\lambda)$-graph $G$ with $d/\lambda\ge C$ contains a crown on $n$ vertices with $\lfloor\epsilon n\rfloor$ spikes.

In the spirit of embedding in the square of a pseudo-random graph, like in our main result here, the following challenge seems quite attractive: given a tree $T$ on $n$ vertices of constant maximum degree $\Delta=O(1)$, prove that if $d/\lambda>C$ for some large enough $C=C(\Delta)$, then the square $G^2$ of an $(n,d,\lambda)$-graph $G$ contains a copy of $T$. The elaborate methods developed by Montgomery in his proof of the embedding result for bounded degree spanning trees in sparse random graphs \cite{Mon19} might be of direct relevance here.

\bigskip

\noindent{\bf Acknowledgement.} The author is grateful to Wojciech Samotij for his remarks.

\end{document}